\documentclass[12pt]{amsart}
\usepackage{amssymb,latexsym,amsmath,amscd,amsthm,graphicx, color}
\usepackage[all]{xy}
\usepackage{pgf,tikz}
\usepackage{mathrsfs}
\usetikzlibrary{arrows}
\definecolor{uuuuuu}{rgb}{0.26666666666666666,0.26666666666666666,0.26666666666666666}
\definecolor{xdxdff}{rgb}{0.49019607843137253,0.49019607843137253,1.}
\definecolor{ffqqqq}{rgb}{1.,0.,0.}

\definecolor{uuuuuu}{rgb}{0.26666666666666666,0.26666666666666666,0.26666666666666666}
\definecolor{qqwuqq}{rgb}{0.,0.39215686274509803,0.}
\definecolor{zzttqq}{rgb}{0.6,0.2,0.}
\definecolor{xdxdff}{rgb}{0.49019607843137253,0.49019607843137253,1.}
\definecolor{qqqqff}{rgb}{0.,0.,1.}
\definecolor{cqcqcq}{rgb}{0.7529411764705882,0.7529411764705882,0.7529411764705882}
\definecolor{sqsqsq}{rgb}{0.12549019607843137,0.12549019607843137,0.12549019607843137}

\theoremstyle{plain}

\newtheorem{theorem}[subsection]{Theorem}

\newtheorem{lemma}[subsection]{Lemma}

\newtheorem{defi}[subsection]{Definition}

\newtheorem{prop}[subsection]{Proposition}
\newtheorem{cor}[subsection]{Corollary}

\theoremstyle{definition}

\newtheorem{note}[subsection]{Note}

\newcommand{\uu}{\cup}
\newcommand{\ii}{\cap}
\newcommand{\UU}{\bigcup}
\newcommand{\II}{\bigcap}

\newcommand{\ci}{\subseteq}
\newcommand{\sci}{\subset}
\newcommand{\es}{\emptyset}
\newcommand{\set}[1]{\{#1\}}

\newcommand{\gd}{\delta}
\renewcommand{\gg}{\gamma}
\newcommand{\gh}{\eta}

\newcommand{\gm}{\mu}

\newcommand{\go}{\omega}

\newcommand{\gs}{\sigma}
\newcommand{\gt}{\tau}
\newcommand{\gx}{\xi}

\newcommand{\gO}{\Omega}

\newcommand{\tbf}{\textbf}
\newcommand{\tit}{\textit}

\newcommand{\C}[1]{\mathcal{#1}}
\newcommand{\D}[1]{\mathbb{#1}}

\newcommand{\te}{\text}

\newcommand{\ep}{\epsilon}

\newcommand{\ol}{\overline}
\newcommand{\ul}{\underline}

\newcommand{\vp}{\varphi}
\begin{document}
To appear, Southeast Asian Bulletin of Mathematics

\title{Topological pressure and fractal dimensions of cookie-cutter-like sets}

\author{Mrinal Kanti Roychowdhury}
\address{School of Mathematical and Statistical Sciences\\
University of Texas Rio Grande Valley\\
1201 West University Drive\\
Edinburg, TX 78539-2999, USA.}
\email{mrinal.roychowdhury@utrgv.edu}


\thanks{The research of the author was supported by U.S. National Security Agency (NSA) Grant H98230-14-1-0320}

\date{}
\maketitle

\pagestyle{myheadings}\markboth{Mrinal Kanti Roychowdhury}{Topological pressure and fractal dimensions of cookie-cutter-like sets}
\begin{abstract}
The cookie-cutter-like set is defined as the limit set of a sequence of classical cookie-cutter mappings. For this cookie-cutter set it is shown that the topological pressure function exists, and that the fractal dimensions such as the Hausdorff dimension, the packing dimension and the box-counting dimension are all equal to the unique zero $h$ of the pressure function. Moreover, it is shown that the $h$-dimensional Hausdorff measure and the $h$-dimensional packing measure are finite and positive.
\end{abstract}

\section{Introduction}
A basic task in Fractal Geometry is to determine or estimate the various dimensions of fractal sets. Fractal dimensions are introduced to measure the sizes of fractal sets and are used in many different disciplines. Many results on fractal dimensions and measures are obtained, among which the studies on self-similar sets are the most rich and thorough, for some details one can see \cite{B, H, M, MM, MU, R, RW, S, YT, YTXY}. In this paper, we have discussed about the topological pressure, fractal dimensions and measures of a typical fractal known as cookie-cutter-like set.
Let $J$ be a nonempty compact subset of $\D R$. For simplicity in our paper we have chosen $J=[0, 1]$. Let $J_1, J_2, \cdots, J_n$,  $n\geq 2$,  be a collection of nonempty closed subsets of $J$. Write \[X:=\UU_{j=1}^k J_j.\] Let $f : X \to J$ be such that each $J_j$ is mapped bijectively onto $J$. We assume that $f$ has a continuous derivative and is expanding so that $|f'(x)|>1$ for all $x\in X$. Let us write
\[E=\set{x \in J : f^k(x) \te{ is defined and in } X \te{ for all } k=0, 1, 2, \cdots},\]
where $f^k$ is the $k$th iterate of $f$. Thus, $E$ is the set of points that remain in $X$ under the iteration of $f$. Since $E=\uu_{k=0}^\infty f^{-k}(J)$ is the intersection of a decreasing sequence of compact sets, the set $E$ is compact and nonempty. $E$ is invariant under $f$, in that
\begin{equation} \label{eq111} f(E)=E=f^{-1}(E),\end{equation}
since $x\in E$ if and only if $f(x) \in E$. Moreover, $E$ is repeller, in the sense that the points not in $E$ are eventually mapped outside of $X$ under iteration by $f$. Let us now define $\vp_j:=(f|_{J_j})^{-1}$ for $1\leq j\leq n$. Then, $\vp_j : J \to J_j$ is a bijection. Since $f$ has a continuous derivative with $|f'(x)|>1$ on the compact set $X$, there are numbers $0<c_{\min}\leq c_{\max}<1$ such that $1<c_{\max}^{-1} \leq |f'(x)|\leq c_{\min}^{-1}<\infty$ for all $x\in X$. It follows that the inverse function $\vp_j$ are differentiable with $0<c_{\min}\leq |\vp_j'(x)|\leq c_{\max}<1$ for $x \in J$. By the mean value theorem, for $1\leq j\leq n$, we have
\[c_{\min} |x-y|\leq |\vp_j(x)-\vp_j(y)|\leq c_{\max}|x-y|\]
for $x, y\in J$. By \eqref{eq111} the repeller $E$ of $f$ satisfies
\[E=\UU_{j=1}^n \vp_j(E).\]
Since each $\vp_j$ is a contraction on $J$, using the fundamental IFS property (see \cite[Theorem 2.6]{F2}), the repeller $E$ of $f$ is the attractor or the invariant set of the set of contraction mappings $\set{\vp_1, \vp_2, \cdots, \vp_n}$. A dynamical system $f:J_1\uu J_2\uu\cdots \uu J_n \to J$ of this form is called a cookie-cutter mapping,  the equivalent iterated function system $\set{\vp_1, \vp_2, \cdots, \vp_n}$ on $J$ is termed a cookie-cutter system and the set $E$ is called a cookie-cutter set. In general, the mappings $\vp_1, \vp_2, \cdots, \vp_n$ are not similarity transformations and $E$ is a `distorted' Cantor set, which nevertheless is `approximately self-similar'. For details about it one could see \cite{F2, B}.  In this paper, we have considered a sequence $\set{f_k}_{k=1}^\infty$ of cookie-cutter mappings, and call the corresponding repeller $E$ the cookie-cutter-like set.

We denote by $\C H^s(E)$, $\C P^s(E)$, $\te{dim}_{\te{H}} E$, $\te{dim}_{\te{P}} E$, $\ul{\te{dim}}_{\te{B}} E$ and $\ol{\te{dim}}_{\te{B}} E$
 the $s$-dimensional Hausdorff measure, the $s$-dimensional packing measure, the Hausdorff dimension, the packing dimension, the lower box-counting and the upper box-counting dimensions of the set $E$, respectively. In this paper, for the cookie-cutter-like set $E$ we have defined the topological pressure function $P(t)$, and showed that there exists a unique $h\in (0, 1)$ such that $P(h)=0$, and then we have defined the Gibbs-like measure $\mu_h$. Using the consequence of the topological pressure and the Gibbs-like measure, we have shown that
 \[\te{dim}_{\te{H}}(E)=\te{dim}_{\te{P}}(E)=\ul{\te{dim}}_{\te{B}}(E)=\ol{\te{dim}}_{\te{B}}(E)=h, \te{ and } 0<\C H^h(E)\leq \C P^h(E) <\infty.\]
The result in this paper is a nonlinear extension, as well as a generalization of the classical result about self-similar sets given by the following theorem (see \cite{H}):

\tbf{Theorem:}  Let $E$ be the self-similar set of the family of similarity contractions $\set{S_1, S_2, \cdots, \\ S_N}$, where $S_j$ has the similarity ratio $c_j$. If the open set condition is satisfied, then
\[\te{dim}_{\te{H}}(E)=\te{dim}_{\te{P}}(E)=\ul{\te{dim}}_{\te{B}}(E)=\ol{\te{dim}}_{\te{B}}(E)=s, \te{ and } 0<\C H^s(E)\leq \C P^s(E) <\infty,\]
where $s$ is the unique positive solution of the equation $\sum_{j=1}^N c_j^s=1$.

\section{Basic results, cookie-cutter-like sets and the topological pressure}
In this section, first we adopt the following definitions and notations, which can be found in \cite{F1, F2}.
Let $E$ be a nonempty bounded subset of $\D R^n$ where $n\geq 1$, and $s\geq 0$.
Then, there are some basic inequalities between the dimensions:
\begin{equation} \label{eq0} \te{dim}_{\te{H}} E \leq \te{dim}_{\te{P}} E \leq \ol{\te{dim}}_{\te{B}} E, \te{ and } \te{dim}_{\te{H}} E \leq \ul{\te{dim}}_{\te{B}} E \leq \ol{\te{dim}}_{\te{B}} E.\end{equation}
Moreover, it is well-known that $\C H^s(E)\leq \C P^s(E)$.
Let $\C U=\set{U_i}$ be a countable collection of sets of $\D R^n$. We define
\[\|\C U\|^s : =\sum_{U_i\in \C U} |U_i|^s,\]
where $|A|$ denotes the diameter of a set $A$. We need the following proposition.
\begin{prop} (see \cite[Proposition~2.2]{F2}) \label{prop0} Let $E\sci \D R^n$ be a Borel set,  $\mu$ be a finite Borel measure on $\D R^n$, and $0<c<\infty$.

$(a)$ If $\limsup_{r\to 0} \mu(B(x, r))/r^s\leq c$ for all $x\in E$, then $\C H^s(E) \geq \mu(E)/c$.

$(b)$ If $\liminf_{r\to 0} \mu(B(x, r))/r^s\geq c$ for all $x\in E$, then $\C P^s(E) \leq 2^s\mu(E)/c$.

\end{prop}

\subsection{Cookie-cutter-like set} A mapping $f$ is called a cookie-cutter, if there  exists a finite collection of disjoint closed intervals $J_1, J_2, \cdots, J_q \sci J =[0, 1]$, such that

$(C1)$ $f$ is defined in a neighborhood of each $J_j$, $1\leq j\leq q$, the restriction of $f$ to each initial interval $J_j$ is 1-1 and onto, the corresponding
branch inverse is denoted by $\vp_j:=(f|_{J_j})^{-1} : J \to J_j$;

$(C2)$ $f$ is differentiable with H\"older continuous derivative $f'$, i.e., there exist constants $c_f>0$ and $\gg_f \in (0, 1]$ such that for $x, y\in J_j$,
$1\leq j\leq q$,
\[\left|f'(x)-f'(y)\right|\leq c_f|x-y|^{\gg_f};\]

$(C3)$ $f$ is boundedly expanding in the sense that there exist constants $b_f$ and $B_f$ such that
\[1<b_f:=\inf_x\set{|f'(x)|} \leq \sup_x\set{|f'(x)| }:=B_f<+\infty.\]
$[\UU_{j=1}^q J_j; c_f, \gg_f, b_f, B_f]$ is called the defining data of the cookie-cutter mapping $f$.

Let $q\geq 2$ be fixed, and consider a sequence of cookie-cutter mappings $\set{f_k}_{k\geq 1}$ with defining data $[\UU_{j=1}^{q}  J_{k, j}; c_k, \gg_k, b_k, B_k].$ Let us write $\vp_{k,j}:=\left(f_k|_{J_{k,j}}\right)^{-1}$ to denote the corresponding branch inverse of $f_k$, where $k\geq 1$ and $1\leq j\leq q$. We always assume that
\[1<\inf \set{b_k}\leq  \sup  \set{B_k} <\infty,  \ 0<\inf\set{\gg_k } \leq \sup\set{\gg_k} \leq 1, \te{ and } 0<\inf \set{c_k} \leq \sup \set{c_k}<\infty.\]
Let $\gO_0$ be the empty set. For $n\geq 1$, define
\[\gO_{n}=\set{1, 2, \cdots, q}^n, \ \gO_\infty=\lim_{n\to \infty} \gO_n, \te{ and } \gO=\UU_{k=0}^\infty \gO_k.\]
Elements of $\gO$ are called words. For any $\gs \in \gO$ if $\gs=(\gs_1, \gs_2, \cdots, \gs_n)  \in \gO_n$, we write $\gs^-=(\gs_1, \gs_2, \cdots, \gs_{n-1})$ to denote the word obtained by deleting the last letter of $\gs$,  $|\gs|=n$ to denote the length of $\gs$, and $\gs|_k:=(\gs_1, \gs_2, \cdots, \gs_k)$, $k\leq n$, to denote the truncation of $\gs$ to the length $k$.  For any two words  $\gs=(\gs_1, \gs_2, \cdots, \gs_k)$ and $\gt=(\gt_1, \gt_2, \cdots, \gt_m)$, we write $\gs\gt=\gs\ast \gt=(\gs_1,  \cdots, \gs_k, \gt_1,  \cdots, \gt_m)$ to denote the juxtaposition of $\gs, \gt \in \gO$.  A word of length zero is called the empty word and is denoted by $\es$. For $\gs \in \gO$ and $\gt\in \gO\uu \gO_\infty$ we say $\gt$ is an extension of $\gs$, written as $\gs\prec \gt$, if $\gt|_{|\gs|}=\gs$. For $\gs=(\gs_1, \gs_2, \cdots, \gs_n) \in \gO_n$, let us write $\vp_\gs=\vp_{1,\gs_1} \circ \cdots \circ\vp_{n,\gs_n}$, and define the rank-$n$ basic interval  corresponding to $\gs$ by
\[J_\gs=J_{(\gs_1, \gs_2, \cdots, \gs_n)}=\vp_{\gs}(J),\]
where $1\leq \gs_k \leq q$, $1\leq k\leq n$. It is easy to see that the set of basic intervals $\set{J_\gs : \gs \in \gO}$ has the following net properties:

$(i)$ $J_{\gs \ast j} \sci J_\gs$ for each $\gs \in \gO_n$ and $1\leq j\leq q$ for all $n\geq 1$;

$(ii)$ $J_{\gs} \II J_{\gt}=\es$, if $\gs, \gt \in \gO_n$ for all $n\geq 1$ and $\gs \neq \gt$.

Let $b=\inf\set{b_k}$ and $B=\sup\set{B_k}$. Then, by our assumption, $1<b\leq B<\infty$. Moreover, by $(C2)$, as $\vp_{k,j}$ is a corresponding branch inverse of $f_k$, where $k\geq 1$ and $1\leq j\leq q$, for all $x \in J$, we have $\left|f_k'(\vp_{k,j}(x))\right|\left|\vp_{k,j}'(x)\right|=1$, and so
\begin{equation} \label{eq1} B^{-1} \leq \left|\vp_{k,j}'(x)\right|\leq b^{-1}. \end{equation}
Next, let
\[E =\II_{n=1}^\infty \UU_{\gs \in \gO_n} J_\gs.\]
Choose $x, y$ to be the end points of $J$, and then $\vp_\gs(x), \vp_\gs(y)$ are the end points of $J_\gs$ for each $\gs \in \gO$, and so by mean value theorem, \[|J_\gs|=|\vp_\gs(x)-\vp_\gs(y)|=|\vp_\gs'(w)||x-y|=|\vp_\gs'(w)|,\]
 for some $w\in J_\gs$. Thus, $B^{-n} \leq |J_\gs| \leq b^{-n}$ for any $\gs \in \gO_n$, and so the diameter $\left|J_\gs\right| \to 0$ as $|\gs|\to \infty$. Hence, with the net properties we conclude that $E$ is a perfect, nowhere dense and totally disconnected subset of $J$. The set $E$ is called the \tit{cookie-cutter-like} (CC-like) set generated by the cookie-cutter sequence $\set{f_k}_{k=1}^\infty$.

\begin{defi} Let $E$ be a CC-like set and $r>0$. The family of basic intervals $\C U_r=\set{J_\gs : |J_\gs|\leq r<|J_{\gs^-}|}\sci \set{J_\gs : \gs \in \gO}$ is called the $r$-Moran covering of $E$ provided it is a covering of $E$, i.e., $E\ci \UU_{J_\gs \in \C U_r}J_\gs$.
\end{defi}
From the definition it follows that the elements of a Moran covering are disjoint, have almost equal sizes, and are often of different ranks.

The following lemma is known.

\begin{lemma} (see \cite[Lemma 2.1]{MRW}) \label{lemma1}
There exists a constant $1<\gx<+\infty$ such that for each $n\geq 1$,  $\gs \in \gO_n$, and $x, y \in J_\gs$, we have
\[\gx^{-1} \leq \frac{|F_n'(x)|}{|F_n'(y)|} \leq \gx,\]
where $F_n(x)=f_n\circ f_{n-1}\circ \cdots \circ f_1(x)$.
\end{lemma}

The following lemma follows from Lemma~2.2 in \cite{MRW}.

\begin{lemma} \label{lemma1111} Let $\gx$ be the constant of Lemma~\ref{lemma1}, and $n, k \in \D N$. Then, for $\go_1, \go_2 \in \gO_n$ and $\gs \in \gO_k$, we have
\[\gx^{-2}\frac{|J_{\go_2\ast\gs}|}{|J_{\go_2}|}\leq \frac{|J_{\go_1\ast\gs}|}{|J_{\go_1}|}\leq \gx^{2}\frac{|J_{\go_2\ast\gs}|}{|J_{\go_2}|}.\]

\end{lemma}

Let us now prove the following proposition.
\begin{prop} \label{prop1}
Let $\gs\in\gO_n$,  $n\geq 1$, $x, y \in J$ and $\gx$ be the constant of Lemma~\ref{lemma1}. Then,
\[\gx^{-1} |\vp_\gs'(y)| \leq |\vp_\gs'(x)| \leq \gx |\vp_\gs'(y)|.\]
\end{prop}
\begin{proof}
For $\gs\in \gO_n$, $n\geq 1$, and $x \in J$, we know $F_n(\vp_\gs(x))=x$. Thus,
\[|F_n'(\vp_\gs(x))|\cdot|\vp_\gs'(x)|=1 ,\te{ and so } |F_n'(\vp_\gs(x))|=\frac 1{|\vp_\gs'(x)|}.\]
Again for all $x \in J$, $\vp_\gs(x) \in J_\gs$. Hence, Lemma~\ref{lemma1} yields
\[\gx^{-1} |\vp_\gs'(y)| \leq |\vp_\gs'(x)| \leq \gx |\vp_\gs'(y)|,\]
and thus, the proposition is obtained.
\end{proof}
For any $\gs \in \gO$, let us write $\|\vp_\gs\|=\sup_{x\in J} |\vp_\gs(x)|$. From the above proposition the following lemma easily follows.
\begin{lemma}\label{lemma3}
Let $\gs, \gt \in \gO$. Then,
\[\gx^{-1} \|\vp_\gs'\|\|\vp_{\gt}'\|\leq \|\vp_{\gs\gt}'\|\leq \|\vp_\gs'\|\|\vp_{\gt}'\|. \]

\end{lemma}

\subsection{Topological pressure}For $t\in \D R$ and  $n\geq 1$, let us write $Z_n(t)=\sum_{\gs \in \gO_n} \|\vp_\gs'\|^t$. Then, for $n, p\geq 1$,
\[Z_{n+p}(t)=\sum_{\gs\in \gO_n} \sum_{\gt\in \in \gO_p}\|\vp_{\gs\gt}'\|^t.\]
By Lemma~\ref{lemma3}, if $t\geq 0$, then
\[ Z_{n+p}(t) \leq Z_n(t)Z_p(t),\]
and if $t<0$, then
\[ Z_{n+p}(t) \leq \gx^{-t}Z_n(t)Z_p(t).\]
Hence, by the standard theory of subadditive sequences, $\lim_{k\to\infty} \frac 1 k \log Z_k(t)$ exists (see \cite[Corollary 1.2]{F2}). Let us denote it by $P(t)$, i.e.,
\begin{equation} \label{eq2} P(t)=\lim_{k\to\infty} \frac 1 k \log \sum_{\gs\in \gO_k}\|\vp_\gs'\|^t.\end{equation}
The above function $P(t)$ is called the \tit{topological pressure} of the CC-like set $E$. Lemma~\ref{lemma11} and Lemma~\ref{lemma12} give some properties of the function $P(t)$.

\begin{lemma} \label{lemma11}
The function $P(t)$ is strictly decreasing, convex and hence continuous on $\D R$.

\end{lemma}
\begin{proof} To prove that $P(t)$ is strictly decreasing let $\gd>0$. Then, by Lemma~\ref{lemma3} and Inequality~\eqref{eq1},
\begin{align*}
P(t+\gd)&=\lim_{k\to\infty} \frac 1 k \log \sum_{\gs\in \gO_k}\|\vp_\gs'\|^{t+\gd}\leq \lim_{k\to\infty} \frac 1 k \log \sum_{\gs\in \gO_k}\|\vp_\gs'\|^t b^{-k\gd}\\
&=P(q, t) -\gd \log b<P(q, t),
\end{align*}
i.e., $P(t)$ is strictly decreasing. For $t_1, t_2 \in \D R$ and  $a_1, a_2 > 0$ with $a_1+a_2=1$, using H\"older's inequality, we have
\begin{align*}
P(q, a_1 t_1+a_2t_2)&= \lim_{n\to \infty}\frac 1 k \log\sum_{\gs \in \gO_k}  \|\vp_\gs'\|^{a_1 t_1+a_2 t_2}=\lim_{k\to \infty}\frac 1k \log\sum_{\gs \in \gO_k}\Big ({\|\vp_\gs'\|}^{t_1}\Big)^{a_1}\Big( \|\vp_\gs'\|^{t_2}\Big)^{a_2}\\
&\leq \lim_{k\to \infty}\frac 1 k \log\Big(\sum_{\gs \in \gO_k}\|\vp_\gs'\|^{t_1}
\Big)^{a_1}\Big(\sum_{\gs \in \gO_k}\|\vp_\gs'\|^{t_2}\Big)^{a_2}\\
&=a_1P(t_1)+a_2 P(t_2),
\end{align*}
i.e., $P(t)$ is convex and hence continuous on $\D R$.
\end{proof}
Let us now prove the following lemma.
\begin{lemma} \label{lemma12}
There exists a unique $h \in (0, 1)$ such that $P(h)=0$.
\end{lemma}

\begin{proof}
By Lemma~\ref{lemma11}, the function $P(t)$ is strictly decreasing and continuous on $\D R$, and so there exists a unique $h \in \D R$ such that $P(h)=0$. Note that
\[
P(0)=\lim_{k\to\infty} \frac 1 k \log \sum_{\gs\in \gO_k}1=\lim_{k\to\infty} \frac 1 k \log q^k=\log q \geq \log 2>0.
\]
Again, by Lemma~\ref{lemma3} and Inequality~\eqref{eq1},
\[P(1)=\lim_{k\to\infty} \frac 1 k \log \|\vp_\gs'\| \leq \lim_{k\to\infty} \frac 1 k \log b^{-kq}=-q \log b<0.\]
Therefore, by the intermediate value theorem, $h\in (0, 1)$ and thus, the lemma is obtained.
\end{proof}
In the next section, we state and prove the main result of the paper.

\section{Main result}
The following theorem gives the main result of the paper.

\begin{theorem}\label{theorem}  Let $E$ be the cookie-cutter-like set associated with the family $\set{f_k}_{k=1}^\infty$ of cookie-cutter mappings, and $h \in (0, 1)$ be such that $P(h)=0$. Then,
\[\te{dim}_{\te{H}}(E)=\te{dim}_{\te{P}}(E)=\ul{\te{dim}}_{\te{B}}(E)=\ol{\te{dim}}_{\te{B}}(E)=h,\]
and
\[0<\C H^h(E)\leq \C P^h(E) <\infty.\]

\end{theorem}

The following proposition is known.

\begin{prop}(see \cite[Proposition 2.6]{MRW}). \label{prop31}   For any $n\geq 1$, $\gs\in \gO_n$, $x \in J_\gs$, we have
\[\gx^{-1} \leq |J_\gs|\cdot |F_n'(x)| \leq \gx,\]
where $F_n(x)=f_n\circ f_{n-1}\circ \cdots \circ f_1(x)$. Moreover, for each $1\leq j\leq q$, we get $\left|J_{\gs\ast j}\right| \geq \gx^{-1} B^{-1}|J_\gs|$, where $\gx$ is the constant of Lemma~\ref{lemma1}.
\end{prop}
From the above proposition, the following lemma easily follows.
\begin{lemma} \label{lemma32}
Let $n\geq 1$, $\gs\in \gO_n$ and $x \in J$. Then,
\[\gx^{-1} |J_\gs| \leq |\vp_\gs'(x)| \leq \gx |J_\gs|,\]
where $\gx$ is the constant of Lemma~\ref{lemma1}.
\end{lemma}

Let us now prove the following lemma.
\begin{lemma}\label{lemma33}
Let $\gs, \gt \in \gO$. Then,
\[\gx^{-3} |J_\gs||J_\gt| \leq |J_{\gs\gt}| \leq \gx^3 |J_\gs||J_\gt|,\]
where $\gx$ is the constant of Lemma~\ref{lemma1}.
\end{lemma}
\begin{proof}
For $\gs, \gt \in \gO$, we have  $|\vp_{\gs\gt}'(x)|=|\vp_\gs'(y)||\vp_\gt'(x)|$, where $y=\vp_\gt(x)$ for $x \in J$. Again, Proposition~\ref{prop1} gives that for any $x, y \in J$, \[\gx^{-1} |\vp_\gs'(y)| \leq |\vp_\gs'(x)| \leq \gx |\vp_\gs'(y)|.\]
Hence, Lemma~\ref{lemma32} implies
\[\gx^{-3} |J_\gs||J_\gt|\leq \gx^{-1} |\vp_\gs'(y)||\vp_\gt'(x)|=\gx^{-1} |\vp_{\gs\gt}'(x)|\leq |J_{\gs\gt}|\leq \gx |\vp_{\gs\gt}'(x)|\leq \gx^{3} |J_\gs||J_\gt|,\]
and thus, the lemma is obtained.
\end{proof}

\begin{note} Lemma~\ref{lemma32} implies that \[\gx^{-1} |J_\gs| \leq \sup_{x \in J}|\vp_\gs'(x)|=\|\vp_\gs'\| \leq \gx |J_\gs|,\] and so the topological pressure $P(t)$ can be rewritten as follows:
\[P(t)=\lim_{k\to \infty} \frac 1 k \log \sum_{\gs\in \gO_k}|J_\gs|^t. \]

\end{note}

Let us now prove the following proposition, which plays a vital role in the paper.
\begin{prop} \label{prop21}
Let $h \in (0, 1)$ be unique such that $P(h)=0$, and let $s_\ast$ and $s^\ast$ be any two arbitrary real numbers with $0<s_\ast<h<s^\ast<1$. Then, for all $n\geq 1$,
\[\gx^{-3}< \sum_{\gs \in \gO_n}|J_\gs|^{s_\ast} \te{ and } \sum_{\gs \in \gO_n}|J_\gs|^{s_\ast}<\gx^3,\]
where $\gx$ is the constant of Lemma~\ref{lemma1}.
\end{prop}
\begin{proof}
Let $s_\ast<h$. As the pressure function $P(t)$ is strictly decreasing, $P(s_\ast)>P(h)=0$. Then, for any positive integer $n$, by Lemma~\ref{lemma33}, we have
\begin{align*}
0&<P(s_\ast) =\lim_{p\to \infty} \frac 1{np} \log \sum_{\go\in \gO_{np}}|J_\go|^{s_\ast}\leq \lim_{p\to \infty} \frac 1{np} \log \gx^{3(p-1)s_\ast} \left(\sum_{\gs\in \gO_{n}}|J_\gs|^{s_\ast}\right)^p,
\end{align*}
which implies \[0< \frac 1n \log \left(\gx^{3s_\ast} \sum_{\gs\in \gO_{n}}|J_\gs|^{s_\ast}\right) \te{ and so } \sum_{\gs\in \gO_{n}}|J_\gs|^{s_\ast}>\gx^{-3s_\ast}>\gx^{-3}. \]
Now if $h<s^\ast$, then $P(s^\ast)<0$ as $P(t)$ is strictly decreasing.  Then, for any positive integer $n$, by Lemma~\ref{lemma33}, we have
\begin{align*}
0>P(s^\ast) =\lim_{p\to \infty} \frac 1{np} \log \sum_{\go\in \gO_{np}}|J_\go|^{s^\ast}\geq \lim_{p\to \infty} \frac 1{np} \log \gx^{-3(p-1)s^\ast} \left(\sum_{\gs\in \gO_{n}}|J_\gs|^{s^\ast}\right)^p,
\end{align*}
which implies \[0>\frac 1n \log \left(\gx^{-3s^\ast}\sum_{\gs\in \gO_{n}}|J_\gs|^{s^\ast}\right) \te{ and so } \sum_{\gs\in \gO_{n}}|J_\gs|^{s^\ast}<\gx^{3s^\ast}<\gx^3. \]
Thus, the proposition is obtained.
\end{proof}
\begin{cor} \label{cor1}
Since  $s_\ast$ and $s^\ast$ be any two arbitrary real numbers with $0<s_\ast<h<s^\ast<1$, from the above proposition it follows that for all $n\geq 1$,
\[\gx^{-3}\leq \sum_{\gs \in \gO_n}|J_\gs|^h\leq \gx^{3}.\]

\end{cor}
The following proposition plays an important role in the rest of the paper.
\begin{prop} \label{prop32} \tbf{(Gibbs-like measure)}
Let $h\in (0, 1)$ be such that $P(h)=0$. Then, there exists a constant $\gh\geq 1$ and a probability measure $\mu_h$ supported by $E$ such that for any $k\geq 1$ and $\go_0\in \gO_k$,
\[\gh^{-1} |J_{\go_0}|^h\leq \mu_h(J_{\go_0}) \leq \gh |J_{\go_0}|^h.\]
\end{prop}
\begin{proof}
Consider a sequence of probability measures $\set{\mu_m}_{m\geq 1}$ supported by $E$ such that for any $\go_0 \in \gO_m$,
\begin{equation} \label{eq123} \mu_m(J_{\go_0})=\frac{|J_{\go_0}|^h}{\sum_{\gs\in \gO_m} |J_\gs|^h}.\end{equation}
More precisely, we can construct $\mu_m$ as follows: First distribute the unit mass among the rank-$m$ basic intervals according to \eqref{eq123}. Inductively, suppose then that we have already distributed the mass of proportion $\mu_m (J_\gh)$ to basic intervals $J_\gh$, where $\gh\in \gO_n$ for some $n\geq m$. Then, redistribute the mass concentrated on $J_\gh$ to each of its $q$ subintervals $J_{\gh\ast j}$ with proportions $\frac{|J_{\gh\ast j}|^h}{\sum_{\ell=1}^q |J_{\gh\ast\ell}|^h}$, $1\leq j\leq q$, i.e.,
 \[\mu_m(J_{\gh\ast j})=\frac{|J_{\gh\ast j}|^h}{\sum_{\ell=1}^q |J_{\gh\ast\ell}|^h}\mu_m(J_\gh).\]
 Repeating the above procedure, we then get the measure $\mu_m$. Now, fix some $m\geq 1$, and take any $k\leq m$ and $\go_0 \in \gO_k$. We want to check the compatibility of the definition of $\mu_m$ on an arbitrary $J_{\go_0}$. From the disjointness of the sets $J_\go$, we obtain

\[\gm_m(J_{\go_0})=\sum_{\gt \in \gO_{m-k}}\mu_m(J_{\go_0\ast\gt}).\]
Then by \eqref{eq123}, we have
\begin{equation} \label{eq124} \sum_{\gs\in \gO_m}|J_\gs|^h\gm_m(J_{\go_0})=\sum_{\gt \in \gO_{m-k}}|J_{\go_0\ast\gt}|^h.\end{equation}
Take any $\gh\in \gO_k$. By Lemma~\ref{lemma1111} we have
\[|J_{\go_0\ast\gt}|\leq \gx^2\frac{|J_{\gh\ast\gt}| |J_{\go_0}|}{|J_\gh|}.\]
Hence by \eqref{eq124},
\[ |J_\gh|^h \sum_{\gs\in \gO_m}|J_\gs|^h\gm_m(J_{\go_0})\leq\gx^{2h} |J_{\go_0}|^h\sum_{\gt \in \gO_{m-k}}|J_{\gh\ast\gt}|^h,\]
which implies
\[ \sum_{\gh\in \gO_k} |J_\gh|^h \sum_{\gs\in \gO_m}|J_\gs|^h\gm_m(J_{\go_0})\leq\gx^{2h} |J_{\go_0}|^h\sum_{\gh\in \gO_k}\sum_{\gt \in \gO_{m-k}}|J_{\gh\ast\gt}|^h=\gx^{2h}|J_{\go_0}|^h\sum_{\gs \in \gO_m}|J_\gs|^h,\]
and thus,
\[\gm_m(J_{\go_0})\leq \gx^{2h}\frac{|J_{\go_0}|^h}{\sum_{\gh\in \gO_k} |J_\gh|^h}.\]
By an analogous discussion as above, we get \[\gx^{-2h}\frac{|J_{\go_0}|^h}{\sum_{\gh\in \gO_k} |J_\gh|^h}\leq \gm_m(J_{\go_0}),\]
and thus, \[\gx^{-2h}\frac{|J_{\go_0}|^h}{\sum_{\gs\in \gO_k} |J_\gs|^h}\leq \gm_m(J_{\go_0})\leq \gx^{2h}\frac{|J_{\go_0}|^h}{\sum_{\gs\in \gO_k} |J_\gs|^h}.\]
Combining it with Corollary~\ref{cor1}, we get $\gx^{-2h-3}|J_{\go_0}|^h\leq \gm_m(J_{\go_0})\leq \gx^{2h+3}|J_{\go_0}|^h$. Take $\gh=\max\set{1, \gx^{2h+3}}$. Then $\gh\geq 1$, and
\begin{equation} \label{eq125} \gh^{-1} |J_{\go_0}|^h\leq \mu_m(J_{\go_0}) \leq \gh |J_{\go_0}|^h.\end{equation}
In the above way, we construct a sequence of probability measures $\set{\mu_m}_{m\geq1}$ which are supported by $E$ and satisfy \eqref{eq125} for any $k\leq m$ and $\go_0 \in \gO_k$. Since all the measures $\mu_m$ are probability measures, we are able to extract a subsequence $\set{\mu_{m_n}}_{n=1}^\infty$ converging weakly to a limit measure $\mu_h$. To verify that $\mu_h$ fulfills the desired requirements we use Theorem~1.24 in \cite{M1}.
Fix some $k\geq 1$ and $\go_0 \in \gO_k$. Then, by the properties of weak convergence, $\limsup_{n\to\infty} \mu_{m_n}(J_{\go_0}) \leq \mu_h(J_{\go_0})$. Combining it with \eqref{eq125}, we get $\gh^{-1} |J_{\go_0}|^h\leq \mu_h(J_{\go_0})$. On the other hand, take any $\ep>0$ small enough so that the $\ep$-neighborhood $J(\ep)$ of $J_{\go_0}$ is separated from the other rank-$k$ basic intervals. Then, $\mu_{m_n}(J(\ep))= \mu_{m_n}(J_{\go_0})$ for all $n\geq 1$. Now, from the properties of weak convergence on open sets (see \cite[Theorem~1.24]{M1}), the following holds:
\[\liminf_{n\to\infty} \mu_{m_n}(J(\ep))\geq \mu_h (J(\ep))\geq \mu_h(J_{\go_0}).\]
Combining it with \eqref{eq125}, we get
$\mu_h(J_{\go_0})\leq \liminf_{n\to\infty} \mu_{m_n}(J_{\go_0})\leq \gh |J_{\go_0}|^h.$
Thus, we have
\begin{equation*} \gh^{-1} |J_{\go_0}|^h\leq \mu_h(J_{\go_0}) \leq \gh |J_{\go_0}|^h.\end{equation*}
Finally, for any $x\not \in E$, since $E$ is closed, there exists an open interval $U$ containing $x$ and separated from $E$ such that $\mu_h(U)\leq \liminf_{n\to \infty} \mu_{m_n}(U)=0$, which asserts that $\mu_h$ is supported by $E$. Hence, the proof of the proposition is complete.
\end{proof}

The following proposition is useful.

\begin{prop} \label{prop111}
Let $r>0$,  and let $\C U_r$ be the $r$-Moran covering of $E$. Then, there exists a positive integer $M$ such that the ball $B(x, r)$  of radius $r$, where $x\in J$, intersects at most $M$ elements of $\C U_r$.
\end{prop}
\begin{proof}
By Proposition~\ref{prop31}, for any $\gs\in \gO$, we get
$\left|J_\gs\right|\geq \gx^{-1} B^{-1}|J_{\gs^-}|$.
Let $\C U_r$ be the $r$-Moran covering of $E$. Fix any $x\in J$, and write $V=B(x, r)$. Define, \[Q_V=\set{J \in \C U_r : J\ii V \neq \es}.\]
Note that any interval $J_\gs$ in $Q_V$ contains a ball of radius $\frac 12 |J_\gs|$, and all such balls are disjoint.
Moreover, all the balls are contained in a ball of radius $2r$ concentric with $V$, and so comparing the volumes (in fact lengths),
\[2r\geq \left(\# Q_V\right)\frac 1  2 |J_\gs|\geq \left(\# Q_V\right) \frac 1 2\gx^{-1}  B^{-1}|J_{\gs^-}|>\left(\# Q_V\right)\frac 12 \gx^{-1}  B^{-1}r,\]
which implies $\#Q_V< 4 \gx B$. Hence, $M:=\lfloor 4 \gx B \rfloor$ fulfills the statement of the proposition.

\end{proof}

\begin{prop} \label{prop121}
Let $h \in (0, 1)$ be such that $P(h)=0$. Then,
\[0<\C H^h(E)<\infty \te{ and }\te{dim}_{\te{H}}(E)=h.\]
\end{prop}

\begin{proof} For any $n\geq 1$, the set $\set{J_\gs : \gs \in \gO_n}$ is a covering of the $E$ and so by Corollary~\ref{cor1},
\[\C H^h(E) \leq \liminf_{n\to \infty} \sum_{\gs \in \gO_n}|J_\gs|^h\leq \gx^{3}<\infty,\]
which yields dim$_{\te{H}}(E)\leq h$.
Let $\mu_h$ be the Gibbs-like measure defined in Proposition~\ref{prop32}, and $k\geq 1$. Then, for any $\gs \in \gO_k$, we get  $\mu_h(J_\gs)\leq \gh |J_\gs|^h$. Let $r>0$ and let $\C U_r=\set{J_\go : |J_{\go}|\leq r<|J_{\go^{-}}|}$ be the $r$-Moran covering of $E$.  Then, by Proposition~\ref{prop111}, we get
\[\mu_h(B(x, r))\leq \sum_{J_{\go} \ii B(x, r) \neq \es}\mu_h(J_{\go}) \leq \gh \sum_{J_{\go} \ii B(x, r) \neq \es}|J_{\go}|^h\leq \gh M r^h.\]
Thus,
\[\limsup_{r\to 0} \frac{\mu_h(B(x, r))}{r^h}\leq \gh M,\]
and so by Proposition~\ref{prop0}, $\C H^h(E) \geq \gh^{-1} M^{-1}>0$, which implies that dim$_{\te{H}}(E) \geq h$.  Thus, the proposition is yielded.
\end{proof}

Let us now prove the following lemma.

\begin{lemma} \label{lemma122}
Let $h \in (0, 1)$ be such that $P(h)=0$. Then, $\ol{\te{dim}}_{\te{B}}(E)\leq h$.
\end{lemma}
\begin{proof}
Let $\mu_h$ be the Gibbs-like measure defined in Proposition~\ref{prop32}, and for $r>0$ let $\C U_r=\set{J_\go : |J_\go|\leq r<|J_{\go^{-}}|}$ be the $r$-Moran covering of $E$. Then, for any $J_\gs \in \C U_r$, we get $\mu_h(J_\gs)\geq \gh^{-1} |J_\gs|^h$. Thus, it follows that
\[\|\C U_r\|^h =\sum_{J_\gs \in \C U_r} |J_\gs|^h \leq \gh \sum_{J_\gs \in \C U_r}\mu_h(J_\gs)= \gh.\]
Again, for any $J_\gs \in \C U_r$ by Proposition~\ref{prop31}, it follows that $|J_\gs|\geq \gx^{-1}B^{-1} |J_{\gs^-}| >\gx^{-1}B^{-1}r$. Hence,
$\left(\gx B\right)^{-h}r^h N_r(E) \leq \|\C U_r\|^h\leq \gh$, where $N_r(E)$ is the smallest number of sets of diameter at most $r$ that can cover $E$, which implies
 $N_r(E)\leq \gh \left(\gx B\right)^h r^{-h}$ and so
 \[\log N_r(E) \leq \log \left[\gh \left(\gx B\right)^h\right] -h\log r,\]
  which yields
\[\ol{\te{dim}}_{\te{B}} E=\limsup_{r\to 0} \frac{\log N_r(E)}{-\log r} \leq  h,\]
and thus, the lemma is obtained.
\end{proof}

Let us now prove the following proposition.
\begin{prop} \label{prop123}
Let $h \in (0, 1)$ be such that $P(h)=0$, and then $\C P^h(E)<\infty$.
\end{prop}
\begin{proof}
For the Gibbs-like measure $\mu_h$, by Proposition~\ref{prop32}, there exists a constant $\gh\geq 1$ such that $\mu_h(J_\gs)\geq \gh^{-1} |J_\gs|^h$ for $\gs \in \gO_k$, $k\geq 1$. Again, Proposition~\ref{prop31} gives that
\[\left|J_{\gs}\right| \geq \gx^{-1} B^{-1}|J_{\gs^-}|.\] Let
$\C U_r=\set{J_\go : |J_\go|\leq r<|J_{\go^{-}}|}$ be the $r$-Moran covering of $E$ for some $r>0$. Let $x\in J_\gs$ for some $J_\gs \in \C U_r$, and then $J_\gs \sci B(x, r)$. Therefore,
\[\mu_h(B(x, r))\geq \mu_h(J_\gs)\geq \gh^{-1} |J_\gs|^h>\gh^{-1}\gx^{-h} B^{-h} r^h,\]
which implies
\[\liminf_{r\to 0} \frac{\mu_h(B(x, r))}{r^h}\geq \gh^{-1}\gx^{-h} B^{-h},\] and so by Proposition~\ref{prop0},
\[\C P^h(E)\leq 2^h \gh\gx^{h} B^{h}<\infty,\]
and thus, the proposition is obtained.
\end{proof}
\subsection*{Proof of Theorem~\ref{theorem}}
Proposition~\ref{prop121} tells us that dim$_{\te{H}}(E)=h$, and Lemma~\ref{lemma122} gives that
$\ol{\te{dim}}_{\te{B}} E\leq h$. Combining these with the inequalities in \eqref{eq0}, we have
\[\te{dim}_{\te{H}}(E)=\te{dim}_{\te{P}}(E)=\ul{\te{dim}}_{\te{B}}(E)=\ol{\te{dim}}_{\te{B}}(E)=h.\]
Again, from Proposition~\ref{prop121} and Proposition~\ref{prop123} it follows that
\[0<\C H^h(E)\leq \C P^h(E)<\infty.\]
Thus, the proof of the theorem is complete.


\begin{thebibliography}{9999}

\bibitem[B]{B} T. Bedford, \emph{Applications of Dynamical Systems Theory to
Fractals: A Study of Cookie-Cutter Cantor Sets}, Netherlands: Kluwer Academic Publishers,
1991: 1-44.


\bibitem[F1]{F1} K.J. Falconer, \emph{Fractal Geometry: Mathematical Foundations and Applications}, Chichester: Wiley, 1990.

\bibitem[F2]{F2} K.J. Falconer, \emph{Techniques in Fractal Geometry}, Chichester: Wiley, 1997.

\bibitem[H]{H} J.E. Hutchinson, \emph{Fractals and self-similarity}, Indiana Univ. Math. J., 1981, 30: 713-747.




\bibitem[MRW]{MRW} J. Ma, H. Rao and Z. Wen, \emph{Dimensions of cookie-cutter-like sets}, Science in China Series A: Mathematics, Volume 44, Number 11, 1400-1412.
\bibitem[M]{M} P.A.P. Moran, \emph{Additive functions of intervals and Hausdorff measure}, Proc. Camb. Philo. Soc., 1946, 42: 15-23.
\bibitem[M1]{M1} P. Mattila, \emph{Geometry of Sets and Measures in Euclidean Spaces}, Cambridge University Press, 1995.

\bibitem[MM]{MM} M.A. Martin \& P. Mattila, \emph{Hausdorff Measures, Holder Continuous
Maps and Self-Similar Fractals}, Math. Poc. Camb.
Philo. Soc., 1993, 114: 37-42.

\bibitem[MU]{MU} R.D. Mauldin \& M. Urbanski, \emph{Dimensions and Measures in Infinite
Iterated Function Systems},  Proc. London Math. Soc.,
1996, 73: 105-154.


\bibitem[R]{R} M.K. Roychowdhury, \emph{Hausdorff and upper box dimension estimate of hyperbolic recurrent sets}, Israel Journal of Mathematics, 201 (2014), 507-523.



\bibitem[RW]{RW} H. Rao \& Z.Y. Wen, \emph{Some studies of a class of self-similar
fractals with overlap structure}, Adv. Appl. Math., 1998, 20: 50-72.

\bibitem[S]{S} A. Schief, \emph{Separation properties of self-similar sets}, Proc. Amer. Math. Soc., 1994, 122: 111-115.

\bibitem[YT]{YT}  Y. Tie,  \emph{Construction of a class of some affined fractals functions}, Southeast Asian Bull. Math, 36 (3) (2012), 427-439.

\bibitem[YTXY]{YTXY}  Y. Tan \&  X. Yang,  \emph{The construction of Lebesque  space-filling curve}, Southeast Asian Bull. Math, 36 (6z) (2012), 883-890.





\end{thebibliography}
\end{document}